\documentclass{amsart}
\usepackage{fullpage, bbm}

	\theoremstyle{plain}
\newtheorem{theorem}{Theorem}

\theoremstyle{definition}\newtheorem{definition}[theorem]{Definition}

\theoremstyle{remark}\newtheorem{remark}[theorem]{Remark}

\newcommand{\EE}{\mathbb{E} }

\newcommand{\RR}{\mathbb{R} }

\newcommand{\ff}{\mathcal{F} }

\author{Michael R. Tehranchi}
\date{\today}
\title{If $B$ and $f(B)$ are Brownian motions, then $f$ is affine}

\begin{document}

\begin{abstract}  It is shown that if   the processes 
$B$ and $f(B)$ are both Brownian motions (without a random time change) then $f$ must be an affine function.
As a by-product of the proof, it is shown that the only functions which
are solutions to both the Laplace
equation and the eikonal equation are  affine.
\end{abstract}

\maketitle

\section{Statement of results}
Suppose that the process $B$ is a Brownian motion and that the function $f$ is affine. 
Then the process $f(B)$ is again a Brownian motion.  This short note proves the converse:
if both $B$ and $f(B)$ are Brownian motions, then $f$ must be affine.

To be precise, we will use the following definition of Brownian motion:
\begin{definition}\label{def:BM}  The continuous process $B = (B_t)_{t \ge 0}$ is called a $n$-dimensional Brownian motion in a filtration 
$\ff = (\ff_t)_{t \ge 0}$ iff there exists a $n$-dimensional vector $b$ and $n\times n$ non-negative definite
matrix $A$  such that for all $0 \le s \le t$
the conditional distribution of the increment $B_{t}-B_s$ given $\ff_s$ is normal with mean $(t-s)b$ and covariance
matrix $(t-s)A$. 

 A Brownian motion is standard iff $B_0=0$, $b=0$ and $A$ is the $n\times n$ identity matrix.
\end{definition}

The main result of this note is this theorem:
\begin{theorem} \label{th:BM} Suppose $B$ is an $n$-dimensional Brownian motion in the filtration $\ff$ with 
non-singular diffusion matrix $A$.   Suppose
the process $f(B) = ( f(B_t))_{t \ge 0}$ is an $m$-dimensional Brownian motion in the same filtration $\ff$ for a measurable function
$f:\RR^n \to \RR^m$.   Then 
$$
f(x) = Px + q
$$
for some $m \times n$ matrix $P$ and $q \in \RR^m$.
\end{theorem}

There are a number of similar results already in the literature.  For instance,  
  Dudley \cite{D} showed that if $B$ is a  {one-dimensional}
standard Brownian and 
$f: \RR \to \RR$ is a continuous function such that  law of the process $f(B)$ is absolutely continuous with respect to
the law of $B$, then necessarily $f(x)= x$ or $f(x)=-x$.  This implies our Theorem \ref{th:BM}
in the case $n=1$.

When $B$ is an $n$-dimensional standard Brownian motion,
    Bernard,   Campbell \&  Davie \cite{BCD} studied functions $f: \RR^n \to \RR^m$
		such that $f(B)$ is a standard Brownian motion up to a random time change.  For
		instance, it is easy to see by the Dambis--Dubins--Schwarz theorem (see, for instance,
		Section 3.4.B of Karatzas \& Shreve's book \cite{KS}) 
		that 
	in the case $m=1$, it is sufficient that $f$ is harmonic with $f(0)=0$.  In particular, we do not allow time change
	in our Theorem \ref{th:BM}, and hence more structure
	is imposed on the function $f$.

 Letac \&  Pradines proved that if $f: \RR^n \to \RR^m$ is such that
$f(x+ \sqrt{t} Z)$ has the normal distribution for all $x \in \RR^n$ and $t \ge 0$, where 
$Z$ is an $n$-dimensional standard normal random vector, then $f$ is necessarily
equal to an affine function almost everywhere.  At first look, it would seem that
Letac \& Pradines's result would imply our Theorem \ref{th:BM} since if $f(B)$ is a Brownian
motion then $f(B_t)$ is normally distributed for all $t \ge 0$.  However,
the implication is not entirely obvious, thanks to the following (perhaps surprising)
result:
\begin{theorem}\label{th:counter}  Let $B$ be an $n$-dimensional standard Brownian motion with $n \ge 2$. 
There exists a continuous non-linear function $g: \RR^n \to \RR^n$  
 such that the random vectors $g(B_t)$ and $B_t$ have the same law for each $t \ge 0$.  
\end{theorem}

Indeed, the reason that Letac \& Pradines's result does not contradict Theorem \ref{th:counter}
above is that they impose normality  for all $x \in\RR^n$, whereas the mean is fixed at $x=B_0=0$
in Theorem \ref{th:counter}.

The idea of the proof of Theorem \ref{th:BM} is simply an application of the following form of Jensen's inequality:
if $G$ is strictly convex and $\int G(x) \ d \mu = G\left( \int x \ d \mu \right)$ for a probability
measure $\mu$, then $\mu$ is a point mass.
 A similar argument yields a related theorem. 
 We will use the notation  $\| \cdot \|$ for the Euclidean norm 
and $\langle \cdot, \cdot \rangle$ the Euclidean inner product on $\RR^n$.

\begin{theorem}\label{th:eik} Let $D \subseteq \RR^n$ be an open, connected set, 
and suppose $u:D \to \RR$ is a classical solution to both the Laplace equation
$$
\Delta u = 0
$$
and  the eikonal equation
$$
\| \nabla u \| = 1.
$$ 
Then $u(x) = \langle p , x \rangle + q$ for some constants
$p \in \RR^n$ and $q \in \RR$, where $\|p\|=1$.
\end{theorem}
  Theorem \ref{th:eik} is contained in Lemma 4.1 of 
the recent paper of Garnica, Palmas \& Ruiz-Hernandez \cite{GPR}.  
Their proof  appeals to methods of differential geometry, while the proof given below
only uses Jensen's inequality.  

\begin{remark}  There is little loss in assuming that $u$ is a classical solution
to the Laplace equation.  Indeed, if $u$ is only assumed to be
locally integrable and a solution to the Laplace equation in the sense of distributions, then $u$ is automatically
infinitely differentiable, and in particular, a classical solution to the Laplace equation.
 See Section 9.3 of Lieb \& Loss's textbook \cite{LL}.
\end{remark}

\section{Proofs}
In this section, we prove the  results presented above.

\begin{proof}[Proof of Theorem \ref{th:BM}] 
Since every component of a vector-valued Brownian motion is a scalar
Brownian motion, it is sufficient to
consider the case $m=1$.    

First we show that $f$ is smooth.
Now since the conditional distribution of $f(B_t)$ given $\ff_0$ is
normal, we can conclude that
$$
\EE\big[  | f( B_t) | \ \big| \ff_0 \big] < \infty
$$
a.s. for all $t \ge 0$.  In particular,  
we have the growth bound  
\begin{equation}\tag{$*$}
x \mapsto f(x) e^{- \epsilon \| x\|^2 } \mbox{ is Lebesgue integrable on } \RR^n
\end{equation}
for all $\epsilon > 0$.  Now since $f(B)$ is a Brownian motion, 
there is a constant $\mu \in \RR$ such that
$$
\EE[ f(B_t) | \ff_s]  = (t-s) \mu + f(B_s),
$$ 
and hence we have the representation
$$
f(x) = - \tau \mu +    \int f( y ) \phi(\tau, x, y)   dy
$$
where 
$$
\phi(\tau, x, y) =(2\pi \tau)^{-n/2} \det(A)^{-1/2} \exp\left( -\frac{1}{2\tau} \langle y-b\tau -x, A^{-1}  (y-b\tau -x) \rangle \right)
$$
is the Brownian transition density.  
 But by the boundedness property ($*$) and the smoothness of 
 $x \mapsto \phi(t, x, y)$  combined with the dominated convergence
theorem,  the function $f$ is differentiable.
Furthermore, its gradient $\nabla f$ has the representation
$$
\nabla f(x) =      \int \nabla f( y ) \phi(\tau, x, y)   dy
$$
and  also satisfies the boundedness property ($*$).  By
iterating this argument, we see that $f$ is infinitely differentiable.

Now we show that $f$ must satisfy an eikonal equation. Note that It\^o's formula says 
$$
df(B_t) =  \langle \nabla f(B_t), dB_t \rangle + \frac{1}{2} \Delta f (B_t)  dt.
$$
  Since $f(B)$ is a Brownian motion, the quadratic variation is
$$
[ f(B) ]_t = \int_0^t \|    \nabla f(B_s)\|^2 ds = \sigma^2 t
$$
for some constant $\sigma \ge 0$.  Hence $\nabla f$ is a solution of the 
eikonal equation
$$
\| \nabla f \| = \sigma.
$$
almost everywhere.  But since $f$ is smooth, it solves the eikonal equation everywhere.
Now note that
$$
\sigma^2 = \| \nabla f(x) \|^2 = \int \left\| \nabla f( y)  \right\|^2 \phi(\tau, x, y) dy.
$$
Since the squared Euclidean norm is strictly convex, Jensen's inequality says that for
every $x$ there exists  a vector $p_x \in \RR^n$, possibly
depending on $x$, such that $\nabla f(y) = p_x$ a.e $y \in \RR^n$.  Since $\nabla f$
is continuous, we must have $\nabla f(y) = p$ for all $y$ and for some constant vector $p$. 
Hence $f(y) = \langle p,  y \rangle + q$ as claimed.
 \end{proof}
 
 We now proceed to the proof of the Theorem \ref{th:eik}.  It follows the same pattern,
 but  it differs in a few details which we spell out for completeness.

\begin{proof}[Proof  of Theorem \ref{th:eik}] 
 We will show that there is a unit vector $p$ such that
$\nabla u(x) = p$ everywhere in $D$.   Below we will use the notation 
$B =\{ x \in \RR^n : \|x \| < 1 \}$ to denote the open unit
ball in $\RR^n$, and hence $x + rB$ denotes the ball of radius $r \ge 0$ centred
at the point $x \in \RR^n$.

 Since $u$ is harmonic,
it is well known (again, see Section 9.3 of \cite{LL}) that $u$ has the mean-value property: for
every constant $r > 0$ such that $x+ r B \subseteq D$ we have
$$
u(x) = \frac{1}{r^n V} \int_{rB} u( x + y) dy  
$$
where 
$$
V = \frac{ \pi^{n/2}}{\Gamma(n/2)}
$$
denotes the Lebesgue measure of the unit ball $B$.
 Since $u$ is 
continuously differentiable in $D$, the gradient $\nabla u$ is bounded on compact sets, so the dominated convergence
theorem allows us to differentiate both sides of the above equation, yielding
$$
\nabla u(x) = \frac{1}{r^n V} \int_{rB} \nabla u( x + y) dy  
$$
Now for each $x \in \RR^n$,  note that
$$
1  = \| \nabla u(x) \|^2 = \frac{1}{r^n V} \int_{rB} \left\|  \nabla u( x + y) \right\|^2 dy.  
$$
Again, since the squared Euclidean norm is strictly convex,
Jensen's inequality says that there is a vector $p_x$, possibly
depending on $x$, such that $\nabla u(z) = p_x$ a.e $z  \in x + rB$ and $\|p_x\| = 1$.  Since $\nabla u$
is continuous, we must have $\nabla u(z) = p_x$ for all $z$ such that $\|x-z \| \le r$.  

Now fix two points $x$ and $x'$ in $D$.  
Since $D$ is open and connected, there exists a path $C \subseteq D$ connecting them.
Hence there exists a finite number of points
$x= x_1, \ldots, x_N = x' \in D$ and radii $r_1, \ldots r_N > 0$ such that
$\{ x_i + r_iB \}_{i=1}^N$ is a cover of the compact set $C \subseteq D$.
In particular, the $p_x = p_{x'}$ and hence $\nabla u$ is constant on $D$ as claimed.  
\end{proof}

Finally, we construct an example of the function $g$ claimed to exist in Theorem \ref{th:counter}.

\begin{proof}[Proof of Theorem \ref{th:counter}]
Let $S^{n-1} = \{ x \in \RR^n: \| x \|= 1\}$ be the unit $(n-1)$-dimensional sphere and let $\lambda$ be
the uniform probability measure  on $S^{n-1}$.
Let $h: S^{n-1} \to S^{n-1}$ be a continuous $\lambda$-preserving transformation. Finally, let
$g(0) = 0$ and 
$$
g(x) = \| x \| \  h \left(   \frac{ x}{ \| x \|} \right)  
$$
when $x \ne 0$.

Fix a bounded and measurable function $\varphi: \RR^n \to \RR$ and  $t \ge 0$. Using 
the assumption that the transformation $h$ preserves the measure $\lambda$, we have
\begin{align*}
\EE[ \varphi \circ g(B_t) ] & = 
\int_{\RR^n} \varphi \left[ \sqrt{t} \| x \| h \left( \frac{ x}{\|x\|} \right) \right] \frac{e^{-\|x\|^2/2}}{(2\pi)^{n/2}} dx  \\
&= \int_0^{\infty} \int_{S^{n-1}} \varphi \left[ \sqrt{t}  r \ h \left(u \right) \right] \frac{r^{n-1} e^{-r^2/2}}{2^{n/2-1} \Gamma(n/2)}   \lambda(du) \ dr   \\
&= \int_0^{\infty} \int_{S^{n-1}} \varphi ( \sqrt{t} r  u   ) \frac{r^{n-1} e^{-r^2/2}}{2^{n/2-1} \Gamma(n/2)}  \lambda(du) \ dr   \\
&= \int_{\RR^n} \varphi ( \sqrt{t} x  ) \frac{e^{-\|x\|^2/2}}{(2\pi)^{n/2}} dx  \\
& = \EE[ \varphi( B_t)]
\end{align*}
where we have used the polar coordinates $x = r u$ with $r \ge 0$ and $u \in S^{n-1}$.   Hence $g(B_t)$ and $B_t$
have the same law for each $t \ge 0$.

 To show that there exist at least one
 function $h$ which is non-linear,
it is sufficient to consider the case $n=2$ since we may restricting attention to the 
first two coordinates of $B$.   Now let $h: S^1 \to S^1$ be defined by
$h(\cos(\theta), \sin(\theta) ) = (\cos(2 \theta), \sin(2 \theta))$.  It is well-known that this transformation $h$
is measure preserving.  Explicitly, the function $g$ in this case is
$$
g(x_1,x_2) = \left( \frac{ x_1^2 - x_2^2}{\sqrt{x_1^2 + x_2^2}}, \frac{ 2 x_1 x_2}{\sqrt{x_1^2 + x_2^2}} \right)
$$
when $x \ne 0$.
\end{proof}

\noindent
\textbf{Acknowledgment.} I would like to thank Pat Fitzsimmons for 
feedback on an early draft of this note, and in particular, for alerting me to several papers in the literature.
 I would also like to thank Chris Rogers and Arun Thillaisundaram for interesting discussions of this problem, and
 the members of the 
Laboratoire de Probabilit\'es et Mod\`eles Al\'eatoires at Universit\'e Paris 6 for their hospitality
during much of the writing of this note.  Finally, I would like to thank an anonymous referee for helpful
comments.


\begin{thebibliography}{9}
\bibitem{BCD} A. Bernard, E.A. Campbell and A.M. Davie. 
Brownian motion and generalized analytic and inner functions. 
\textit{Annales de l'Institut Fourier (Grenoble)} \textit{29}(1): xvi, 207–-228 (1979) 

\bibitem{D} R. Dudley. Non-linear equivalence transformations of Brownian motion.
\textit{Zeitschrift f\"ur Wahrscheinlichkeitstheorie und Verwandte Gebiete} \textbf{20}:  249--258 (1971)

\bibitem{GPR} 
E. Garnica, O. Palmas and G. Ruiz-Hernandez. 
Classification of constant angle hypersurfaces in warped products via eikonal
functions. \textit{Sociedad Matem\'atica Mexicana. Bolet\'\i n. Tercera Serie}
\textbf{18}(1): 29--41 (2012) 

\bibitem{KS} I. Karatzas and S. Shreve. \textit{Brownian Motion and Stochastic Calculus}.
Graduate Texts in Mathematics \textbf{113}. Springer (1991)

\bibitem{LP} G. Letac and J. Pradines. Seules les affinit\'es pr\'eservent les lois normales. 
\textit{Comptes Rendus Hebdomadaires des S\'eances de l'Acad\'emie des Sciences, S\'erie A.}  \textbf{286}(8): 399-402   (1978)  

\bibitem{LL} E. Lieb and M. Loss. \textit{Analysis}.  Graduate Studies in Mathematics \textbf{14}.
American Mathematics Association  (2001)
 

\end{thebibliography}
\end{document}